\documentclass[a4paper, 12pt]{amsart}

\usepackage{amsthm}
\usepackage{amssymb}
\usepackage{mathrsfs}
\usepackage{latexsym}
\usepackage{stmaryrd}
\usepackage{cases}
\usepackage{amscd}
\usepackage{amsfonts}
\usepackage{txfonts}
\usepackage{amsmath}
\usepackage[all]{xy}
\usepackage{xy}
\usepackage{enumerate}

\usepackage{color}
\definecolor{darkblue}{rgb}{0,0,0.9}
\definecolor{lightblue}{rgb}{.5,.5,.9}
\usepackage[pdftex,breaklinks,colorlinks,filecolor=blue,linkcolor=blue,citecolor=blue,urlcolor=blue,hypertexnames=true,plainpages=false]{hyperref}

\usepackage{tikz}
\usepackage{etoolbox}
\usepackage{comment}
\usepackage{graphics}
\include{definitionsart}

\numberwithin{equation}{section}
\newcommand\C{\mathbb{C}}

\newcommand\R{\mathbb{R}}

\newcommand\Z{\mathbb{Z}}

\theoremstyle{plain}
\newtheorem{theorem}{Theorem}[section]
\newtheorem{lemma}[theorem]{Lemma}
\newtheorem{corollary}[theorem]{Corollary}
\newtheorem{proposition}[theorem]{Proposition}

\theoremstyle{definition}

\newtheorem{example}[theorem]{Example}

\theoremstyle{remark}
\newtheorem{remark}[theorem]{Remark}

\usepackage{color}

\advance \topmargin by -\headheight
\advance \topmargin by -\headsep
     
\evensidemargin \oddsidemargin
\makeatletter
\let\uppercasenonmath\@gobble
\makeatother

\title[Stable almost complex structures on certain $10$-manifolds]{Stable almost complex structures on certain $10$-manifolds}

\author[Huijun Yang]{Huijun Yang}
\address{School of Mathematics and Statistics, Henan University, Kaifeng 475004, Henan, China}
\email{yhj@amss.ac.cn}

\thanks{The author is partially supported by the National Natural Science Foundation of China (Grant No.11301145) and the China Scholarship Council (File No. 201708410052). }


\date{\today}

\begin{document}

\maketitle




\begin{abstract}
Let $M$ be a $10$-dimensional closed oriented smooth manifold. Set
$$\mathcal{D}_{M} := \{ x \in H^{2}(M; \Z/2) \mid x^{2} + w_{2}(M) x \in \rho_{2} ( TH^{4}(M;\Z) ) \}.$$
Suppose that $H_{1}(M;\Z)=0$ and $\mathcal{D}_{M} \subset \rho_{2}( H^{2}(M; \Z) )$. 
Then the necessary and sufficient conditions for $M$ to admit a stable almost complex structure are determined in terms of the characteristic classes and cohomology ring of $M$.
\end{abstract}




\section{Introduction} 
\label{s:intro}



For a topological space $X$. 
Let $\xi$ be a real vector bundle over $X$.
We say that $\xi$ admits a stable complex structure, 
if there exists a complex vector bundle $\eta$ over $X$ such that the underlying real vector bundle $\eta_{\R}$ is stably isomorphic to $\xi$.
For a manifold $M$, we say that $M$ admits a stable almost complex structure, if its tangent bundle $TM$ admits a stable complex structure.


It is a classical topic in geometry and topology to determine the necessary and sufficient conditions (in terms of the characteristic classes) for a $n$-manifold $M$ to admit a stable almost complex structure (see for instance \cite{wu,eh,ma,thomas,heaps,gms,tang,young12, young15} and so on).
These are only known in the case $n\le8$ (cf. \cite{wu, eh, heaps}, etc.). Suppose that $n=10$. Then these conditions for $M$ to admit a stable almost complex structure are determined by Thomas in the case $H_{1}(M;\Z/2)=0$ and $w_{4}(M)=0$ (cf. \cite[Theorem 1.6]{thomas} ). Moreover, these conditions are got by Dessai \cite[Theorems 1.2]{dessai} in the case $H_{1}(M;\Z)=0$ and $H_{i}(M;\Z), i=2,3$ has no $2$-torsion. 

In this paper, our main results are stated as follows.

From now on, $M$ will be a closed oriented smooth $10$-manifold.
We will denote by 
\begin{align}\label{eq:beta}
\cdots\rightarrow H^{i}(M:\Z)\xrightarrow{\rho_{2}}H^{i}(M;\Z/2)\xrightarrow{\beta}H^{i+1}(M;\Z)\rightarrow\cdots
\end{align}
the long exact Bockstein sequence associated to the coefficient sequence
\[0\rightarrow \mathbb{Z}\xrightarrow{\times2}\mathbb{Z}\rightarrow\mathbb{Z}/2\rightarrow0,\]
where  
$\rho_{2}$ and $\beta$ are
the mod $2$ reduction and
Bockstein homomorphism respectively. Denote by $w_{i}(M)$ the $i$-th Stiefel-Whitney class of $M$.
Let $TH^{i}(M; \Z)$ be the torsion subgroup of $H^{i}(M; \Z)$. Set 
\[ \mathcal{D}_{M} := \{ x \in H^{2}(M; \Z/2) \mid x^{2} + w_{2}(M) x \in \rho_{2} ( TH^{4}(M;\Z) ) \}.\]
One may find that $\mathcal{D}_{M}$ is a subgroup of $H^{2}(M;\Z/2)$ and $w_{2}(M) \in \mathcal{D}_{M}$.

Suppose that $\mathcal{D}_{M} \subset \rho_{2} ( H^{2}(M; \Z) )$. That is, for any $x \in \mathcal{D}_{M}$, 
there exist $\tilde{x} \in H^{2}(M; \Z)$ such that $\rho_{2}(\tilde{x}) = x$.
In particular, there exists $c \in H^{2}(M; \Z)$ such that $\rho_{2}(c) = w_{2}(M)$, i.e., $M$ is spin$^{c}$. 
Hence, for any $x \in \mathcal{D}_{M}$, there exists $z_{c, \tilde{x}} \in H^{4}(M; \Z)$ and $t_{c, \tilde{x}} \in TH^{4}(M; \Z)$ such that 
\[ \tilde{x}^{2} + c \tilde{x} = 2z_{c,\tilde{x}} + t_{c,\tilde{x}} \]
by the definition of $\mathcal{D}_{M}$.

\begin{theorem}\label{thm:main}
Let $M$ be a $10$-manifold with $H_{1}(M; \Z) = 0$ and $\mathcal{D}_{M} \subset \rho_{2} ( H^{2}(M; \Z) )$. Then $M$ admits a stable almost complex structure if and only if $\beta (w_{6}(M))=0$ and 
\begin{equation}\label{eq:sacs}
w_{4}^{2}(M)\cdot x = \rho_{2}(z_{c, \tilde{x}}) \cdot w_{6}(M)
\end{equation}
holds for every $x \in \mathcal{D}_{M}$.
\end{theorem}

\begin{remark} 
It follows from the proof of Theorem \ref{thm:main}, Lemma \ref{lem:ac} and Remark \ref{rem:acxiw} in section \ref{s:proof} that the value $\rho_{2}(z_{c, \tilde{x}}) \cdot w_{6}(M)$ does not depend on the choice of $c$ and $\tilde{x}$ and the selection of $z_{c, \tilde{x}}$.
\end{remark}

\begin{remark}
For a $10$-manifold $M$, if $H_{2}(M; \Z)$ contains no $2$-torsion, then $\mathcal{D}_{M} \subset \rho_{2}(H^{2}(M; \Z))$. 
That is because, if $H_{2}(M; \Z)$ contains no $2$-torsion, then so is $H^{3}(M; \Z)$ by the universal coefficient theorem. 
Therefore, it follows from the definition of $\mathcal{D}_{M}$ and the Bockstein sequence \eqref{eq:beta}  that $H^{2}(M; \Z/2) = \rho_{2}(H^{2}(M; \Z))$. 
Thus the claim is proved.
\end{remark}

\begin{remark}
Let $M$ be a $10$-manifold with $H_{1}(M; \Z)=0$ and no $2$-torsion in $H_{i}(M; \Z)$ for $i = 2, ~3$. 
Then $H^{4}(M; \Z)$ contains no $2$-torsion by the universal coefficient theorem. 
Hence $\rho_{2} ( TH^{4}(M; \Z) ) = 0$ and $\mathcal{D}_{M} = \rho_{2} ( \mathcal{D}(M) )$, 
where $\mathcal{D}(M)$ is the subgroup of $H^{2}(M; \Z)$ defined by Dessai in \cite{dessai} (cf. \cite[Difinition 1.1]{dessai}). 
Moreover, under these conditions, one may find from the proof of Theorem \ref{thm:main} that the equation \eqref{eq:sacs} is just the simplification of the congruence $(1.2)$ in \cite{dessai}.
\end{remark}


\begin{remark}
There does exist spin $10$-manifold $M$ for which $\beta (w_{6}(M)) \neq0$ (cf. Diaconescu et al. \cite{dmw02}). However, I do not know of an example for which satisfying that  $H_{1}(M;\Z)=0$ and $\mathcal{D}_{M} \subset \rho_{2} ( H^{2}(M; \Z) )$ as well.
\end{remark}

As applications, note that $\mathcal{D}_{M}$ is the annihilator of $\mathrm{Sq}^{2}\rho_{2} (H^{6}(M;\Z))$ (see Lemma \ref{lem:dm} below) and $w_{6}(M) = \mathrm{Sq}^{2}w_{4}(M)$ (see the identity \eqref{eq:w6} below), one can get

\begin{corollary}\label{coro:dm0}
Let $M$ be a $10$-manifold with $H_{1}(M; \Z) = 0$. Suppose that $\mathcal{D}_{M} \subset \rho_{2}(TH^{2}(M; \Z))$. Then $M$ admits a stable almost complex structure if and only if $\beta (w_{6}(M))=0$.
\end{corollary}

\begin{corollary}\label{coro:H4t}
Let $M$ be as in Theorem \ref{thm:main}. Suppose that $H^{4}(M; \Z)$ is a torsion group. 
Then $M$ admits a stable almost complex structure if and only if $\beta(w_{6}(M)) = 0$.
\end{corollary}

\begin{corollary}\label{coro:w6t}
Let $M$ be as in Theorem \ref{thm:main}. Suppose that $w_{6}(M) \in \rho_{2} ( TH^{6}(M;\Z) )$. 
Then $M$ admits a stable almost complex structure if and only if 
\begin{equation*}\label{eq:w6t}
w_{4}^{2}(M)\in \mathrm{Sq}^{2}\rho_{2} ( H^{6}(M;\Z) ).
\end{equation*}
\end{corollary}



\begin{corollary}\label{coro:w40}
Let $M$ be as in Theorem \ref{thm:main}. Suppose that $w_{4}(M)=0$. Then $M$ always admits a stable almost complex structure.
\end{corollary}

\begin{remark}
This corollary can also be regard as a corollary of Thomas \cite[Theorem 1.6]{thomas}.
\end{remark}

The proofs of Corollaries \ref{coro:dm0}$-$\ref{coro:w40} are left to the reader.

%
%

In fact, in this paper, we do more than above. In section \ref{s:obs}, for any real vector bundle $\xi$ over $M$, the obstructions for $\xi$ to admit a stable complex structure are investigated. Based on the results in section \ref{s:obs}, the necessary and sufficient conditions for $\xi$ to admit a stable complex structure are determined in section \ref{s:proof}.  As an application, with a carefully analyzing  of the Stiefel-Whitney classes of $M$, Theorem \ref{thm:main} is obtained.


\section{The Obstructions}  \label{s:obs}

In this section, $M$ will be a $10$-manifold with $H_{1}(M; \Z) = 0$.
For real vector bundle $\xi$ over $M$, the obstructions for $\xi$ to admit a stable complex structure are investigated.

Let $\mathrm{U}$ (resp.~$\mathrm{SO}$) be the stable unitary (resp. special orthogonal) group. For a pathwise connected $CW$-complex $X$, denote by $X^{(q)}$ the $q$-skeleton of $X$.
Let $\xi$ be a orientable real vector bundle over $X$. Suppose that $\xi$ admits a stable complex structure $\eta$ over $X^{(q)}$.
Then the obstruction $\mathfrak{o}_{q+1}(\eta)$ to extending $\eta$ over $X^{(q+1)}$ lives in $H^{q+1}(X,\pi_{q}(\mathrm{SO/U}))$
where (cf. Bott \cite{bo} or Massey \cite[p.560]{ma})
\begin{equation*}\pi_{q}(\mathrm{SO/U})=
\begin{cases}
\mathbb{Z}, & q\equiv 2\mod{4},\\
\mathbb{Z}/2, & q\equiv 0,~-1\mod{8},\\
0, & otherwise.
\end{cases}
\end{equation*}
It follows from Massey \cite[Theorem I]{ma} that
\begin{lemma}\label{lem:o37}
$\mathfrak{o}_{3}(\eta)=\beta (w_{2}(\xi))$, $\mathfrak{o}_{7}(\eta)=\beta (w_{6}(\xi))$.
\end{lemma}

Recall that for a spin vector bundle $\zeta$ over $X$, the spin characteristic classes $q_{i}(\zeta)\in H^{4i}(X;\Z)$ of $\zeta$ are defined (cf. Thomas \cite{thomas62}) and they satisfy the following relations with the Pontrjagin classes $p_{i}(\zeta)$ and Stifel-Whitney classes $w_{i}(\zeta)$ (cf. \cite[Theorem (1.2)]{thomas62}): 
\begin{align}
p_{1}(\zeta)&=2q_{1}(\zeta), &\rho_{2}(q_{1}(\zeta))=w_{4}(\zeta),\label{eq:q1}\\
p_{2}(\zeta)&=2q_{2}(\zeta)+q_{1}^{2}(\zeta), &\rho_{2}(q_{2}(\zeta))=w_{8}(\zeta).\label{eq:q2}
\end{align}
Moreover, given two spin vector bundles $\zeta_{1}$ and $\zeta_{2}$ over $X$, we have (cf. Thomas \cite[(1.10)]{thomas62})
\begin{align*}
q_{1}(\zeta_{1}\oplus \zeta_{2})&=q_{1}(\zeta_{1})+q_{1}(\zeta_{2}),\\
q_{2}(\zeta_{1}\oplus \zeta_{2})&=q_{2}(\zeta_{1})+q_{2}(\zeta_{2})+q_{1}(\zeta_{1})q_{1}(\zeta_{2}).
\end{align*}
In particular, suppose that $\zeta_{1}-\zeta_{2}$ is stably trivial over $X^{(7)}$, which implies that $q_{1}(\zeta_{1}-\zeta_{2})=0$, then:
\begin{align}
q_{1}(\zeta_{1})&=q_{1}(\zeta_{2}), \notag\\
q_{2}(\zeta_{1}-\zeta_{2})&=q_{2}(\zeta_{1})-q_{2}(\zeta_{2}). \label{eq:mq2}
\end{align}
Here and hereafter, $\zeta_{1} - \zeta_{2} = \zeta_{1} \oplus (- \zeta_{2})$ and $-\zeta_{2}$ is a vector bundle over $M$ such that $\zeta_{2} \oplus (-\zeta_{2}) $ is trivial.
Futhermore, suppose that $\eta$ is a complex vector bundle over $X$ with $c_{1}(\eta)=0$. Then the underlying real vector bundle $\eta_{\R}$ is spin and  it follows from the identities \eqref{eq:q1} and \eqref{eq:q2} that 
\begin{align}\label{eq:cq}
q_{1}(\eta_{\R}) = - c_{2}(\eta),~q_{2}(\eta_{\R}) = c_{4}(\eta).
\end{align}

\begin{proposition}\label{prop:spin}
Let $\zeta$ be a real vector bundle over $X$ which is stably trivial over $X^{(7)}$. Suppose that $q_{2}(\zeta)=0$. Then $\zeta$ must admits a stable complex structure over $X^{(8)}$. 
\end{proposition}

\begin{proof}
Denote by $\mathrm{BO}$ the classifying space of the stable orthogonal group, $\mathrm{BSpin}$ the classifying space of the stable spin group,
$\mathrm{BO}\langle k\rangle$ the ($k {-} 1$)-connected cover of $\mathrm{BO}$ and $p_{k}\colon \mathrm{BO}\langle k\rangle\rightarrow \mathrm{BO}$ the map which satisfies that $p_{k\ast}\colon \pi_{i}(\mathrm{BO}\langle k\rangle)\rightarrow \pi_{i}(\mathrm{BO})$ is isomorphic for $i\ge k$. It is known that $\mathrm{BSpin}=\mathrm{BO}\langle4\rangle$ and there is a map $p\colon \mathrm{BO}\langle8\rangle\rightarrow \mathrm{BSpin}$ such that $p_{4}\circ p=p_{8}$. Denote by $\gamma$ the universal bundle over $\mathrm{BO}$.

Let $\zeta\colon X\rightarrow \mathrm{BO}$ be the map such that $\zeta^{\ast}(\gamma)=\zeta$. 
Since $\zeta$ is stable trivial over $X^{7}$, there exists a map $\zeta^{\prime}\colon X\rightarrow \mathrm{BO}\langle 8\rangle$ such that $p_{8}\circ\zeta^{\prime}=\zeta$. 
Let
\begin{equation*}
q_{2}\colon \mathrm{BO}\langle8\rangle\rightarrow K(\Z;8)
\end{equation*}
be the map corresponding to $p^{\ast}(q_{2})\in H^{8}(\mathrm{BO}\langle8\rangle;\Z)$ where $q_{2}\in H^{8}(\mathrm{BSpin};\Z)$ is the second universal spin characteristic class, $K(\Z;8)$ is the Eilenberg-MacLane space. 
Denote by $F$ the fiber of the map $q_{2}$ and $i\colon F\rightarrow \mathrm{BO}\langle8\rangle$ the inclusion of the fiber. 
Then it follows easily from the relations \eqref{eq:q2} that $F$ is $6$-connected, $\pi_{7}(F)\cong\Z/3$ and $\pi_{8}(F)=0$. 
Therefore, $H^{8}(F;\Z/2)=0$ by \cite[Theorem 5.8]{hass}. Hence, according to the obstruction theory as above, $i^{\ast}p_{8}^{\ast}(\gamma)$ must admits a stable complex structures over the $8$-skeleton $F^{(8)}$ of $F$. 

Now $q_{2}(\zeta)=0$ implies that the map $\zeta^{\prime}$ can factor through $F$ via the inclusion $i$, hence $\zeta^{\prime}$ admits a stable complex structure, and so is $\zeta$. The proof is completed.
\end{proof}


Let $\xi$ be a real vector bundle over $M$. Then $\xi$ must be orientable ( since $H_{1}(M;\Z)=0$) and
there are four obstructions $\mathfrak{o}_{3}(\eta)$, $\mathfrak{o}_{7}(\eta)$, $\mathfrak{o}_{8}(\eta)$ and $\mathfrak{o}_{9}(\eta)$ for $\xi$ to admit a stable complex structure. 
Since $H^{9}(M;\Z/2)\cong H_{1}(M;\Z/2)=0$, it follows from Lemma \ref{lem:o37} and the definition of the obstructions that
\begin{lemma}\label{lem:o90}
The real vector bundle $\xi$ admits a stable complex structure over $M^{(7)}$ if and only if $\beta(w_{2}(\xi) ) = 0$ and $\beta (w_{6}(\xi))=0$. Any stable complex structure $\eta$ of $\xi$ over $M^{(8)}$ can be extended to one over $M$.
\end{lemma}

Furthermore, in the stable range, note that the only obstruction for an extension of a complex vector bundle over $M^{(7)}$ to $M$ lives in $H^{9}(M;\Z)$, we can get that
\begin{lemma}\label{lem:oc}
In the stable range, any complex vector bundle over $M^{(7)}$ can be extended to one over $M$.
\end{lemma}

Now, the main result of this section can be stated as:

\begin{theorem}\label{thm:po8}
Let $\xi$ be a real vector bundle over $M$. 
Then $\xi$ admits a stable complex structure over $M$ if and only if  
there exists a complex vector bundle $\eta$ over $M$, such that $\eta_{\R} |_{M^{(7)}}$ is a stable complex structure of $\xi$ over $M^{(7)}$ and 
\begin{equation}\label{eq:po8w}
w_{8}(\eta) - w_{8}(\xi) \in \mathrm{Sq}^{2}\rho_{2}H^{6}(M;\Z).
\end{equation}
\end{theorem}


\begin{remark}
This is a generalization of Dessai \cite[Lemma 1.7]{dessai}.
\end{remark}

\begin{remark}
If $\eta_{\R} |_{M^{(7)}}$ is a stable complex structure of $\xi$ over $M^{(7)}$, then $\eta_{\R} - \xi $ is stably trivial over $M^{(7)}$, hence a spin vector bundle.  
Therefore, the identities \eqref{eq:q2} implies that
\begin{equation}\label{eq:po8q}
w_{8}(\eta)-w_{8}(\xi) = w_{8}(\eta_{\R} - \xi) = \rho_{2} ( q_{2} ( \eta_{\R} - \xi) ).
\end{equation}
\end{remark}

\begin{proof}
One direction is trivial. 
So assume that there exists a complex vector bundle $\eta$ over $M$, such that $\eta_{\R} |_{M^{(7)}}$ is a stable complex structure of $\xi$ over $M^{(7)}$ and 
$w_{8}(\eta) - w_{8}(\xi) \in \mathrm{Sq}^{2}\rho_{2}H^{6}(M;\Z)$.


Let $\zeta = \eta_{\R} - \xi$. It suffices to show that $\zeta$ admits a stable complex structure. Since $\zeta$ is stably trivial over $M^{(7)}$, which implies that $q_{1}(\zeta)=0$, there must exists $v\in H^{8}(M;\Z)$ such that $q_{2}(\zeta)=3v$ (cf. Duan \cite[Theorem 1]{dl91}). Now it follows easily from the condition \eqref{eq:po8w} and eqution \eqref{eq:po8q} that there is $u\in H^{6}(M;\Z)$ such that 
\[\mathrm{Sq}^{2}\rho_{2}(u)=\rho_{2}(q_{2}(\zeta))=\rho_{2}(v).\]
Then according to Adams \cite{ad61} (cf. Heaps \cite[p. 112]{heaps}), there exists a complex vector bundle $\gamma$ over $M$, trivial over $M^{(5)}$, such that $c_{3}(\gamma)=2u$ and $c_{4}(\gamma)=q_{2}(\zeta)$. Note that we have $\pi_{i}(\mathrm{BO})=0$ for $i=6,7$. 
Therefore, $\gamma_{\R}$ is stably trivial over $M^{(7)}$ and $q_{2}(\gamma_{\R}) = c_{4}(\gamma) = q_{2}(\zeta)$ by identity \eqref{eq:cq}. 

Now, it suffices to show that $\zeta - \gamma_{\R}$ admits a stable complex structure. Since $\zeta - \gamma_{\R}$ is stably trivial over $M^{(7)}$ and $q_{2}(\zeta - \gamma_{\R}) =0 $ by identity \eqref{eq:mq2}, it follows from Proposition \ref{prop:spin} that $\zeta - \gamma_{R}$ admits a stable complex structure over $M^{(8)}$ and hence over $M$ by Lemma \ref{lem:o90}. 
This completes the proof.
\end{proof}


\section{The proof of the main results}  \label{s:proof}
In this section, $M$ will be a $10$-manifold with $H_{1}(M;\Z)=0$ and $\mathcal{D}_{M} \subset \rho_{2}(H^{2}(M; \Z))$.
For a real vector bundle $\xi$ over $M$, based on Theorem \ref{thm:po8}, the necessary and sufficient conditions for $\xi$ to admit a stable complex structure are given in terms of the characteristic classes of $\xi$ and $M$, and the cohomology of $M$. As an application, the proof of Theorem \ref{thm:main} is given.

Recall that $w_{2}(M) \in \mathcal{D}_{M}$. Then $\mathcal{D}_{M} \subset \rho_{2}(H^{2}(M; \Z))$ implies that $M$ is Spin$^{c}$.  
We will fix an element $c\in H^{2}(M;\Z)$ such that $\rho_{2} (c) = w_{2}(M)$.
For a complex vector bundle $\eta$ over $M$, 
let
\begin{align}\label{eq:cheta}
ch(\eta)  = & ~\dim_{\C} \eta + c_{1}(\eta) + \frac{c_{1}^{2}(\eta) - 2c_{2}(\eta)}{2} + \frac{c_{1}^{3}(\eta) - 3c_{1}(\eta)c_{2}(\eta) + 3c_{3}(\eta)}{6}  \\
 & ~ + \frac{c_{1}^{4}(\eta) - 4c_{1}^{2}(\eta)c_{2}(\eta) + 2c_{2}^{2}(\eta) + 4c_{1}(\eta)c_{3}(\eta) - 4c_{4}(\eta)}{24} + ch_{5}(\eta),  \notag
\end{align}
be the Chern character of  $\eta$.
Denote by $[M]$ the fundamental class of $M$, $\langle ~\cdot~,~\cdot~ \rangle$ the Kronecker product and 
$\hat{\mathfrak{A}}(M)$ the $\mathfrak{A}$ class of $M$. It is known that
$$\hat{\mathfrak{A}}(M)=1-\frac{p_{1}(M)}{24}+\frac{-4p_{2}(M)+7p_{1}^{2}(M)}{5760}.$$
Then the differential Riemann-Roch theorem (cf. \cite[Corollary 1]{ah}) tells us that

\begin{lemma}\label{lem:rr}
For any complex vector bundle $\eta$ over $M$, the rational number 
\begin{equation*}
\langle \hat{\mathfrak{A}}(M)\cdot e^{c/2}\cdot ch(\eta),~[M]\rangle
\end{equation*}
is an integer.
\end{lemma}

For any $x\in H^{2}(M;\Z)$, we will denote by $l_{x}$ the complex line bundle with $c_{1}(l_{x})=x$. 
Moreover, for a complex vector bundle $\eta$ over $M$,  we will define
\begin{equation*}
\hat{A}(c, x, \eta) = \langle \hat{\mathfrak{A}}(M) \cdot e^{c/2}\cdot [ch(l_{x}) - 1] \cdot [ch(\eta)- \dim_{\C} \eta], ~[M] \rangle,
\end{equation*}
which is an integer by Lemma \ref{lem:rr}. 
Obviously, for complex vector bundles $\eta$ and $\gamma$ over $M$, we have
\begin{equation*}
\hat{A}(c, x, \eta \pm \gamma) = \hat{A}(c, x, \eta) \pm \hat{A}(c, x,  \gamma).
\end{equation*}

\begin{example}\label{exam:eta}
For a complex vector bundle $\eta$ over $M$, 
denote by $\bar{\eta}$ the complex conjugation of $\eta$ and set 
$$l_{\eta}=l_{c_{1}(\eta)}, \quad \eta^{\prime} = \eta - l_{\eta}. $$
Then $c_{1}(\eta^{\prime}) = 0$. 
Hence
$$ch(\eta^{\prime}) - ch(\bar{\eta^{\prime}} ) = c_{3}(\eta^{\prime}) + 2 ch_{5}(\eta^{\prime})$$
by the equation \eqref{eq:cheta}.
Therefore, 
\begin{align*}
\hat{A}(c, x, \eta^{\prime}) - \hat{A}(c, x, \bar{\eta^{\prime}}) & = 
\langle \hat{\mathfrak{A}}(M) \cdot e^{c/2}\cdot [ch(l_{x}) - 1] \cdot [ch(\eta^{\prime})- ch(\bar{\eta^{\prime}})],~[M] \rangle \\
& = \langle c_{3}(\eta^{\prime}) \cdot (x^{2}+cx)/2,~[M] \rangle,
\end{align*}
which is an integer by Lemma \ref{lem:rr}. \qed 
\end{example}

\begin{example}\label{exam:xi}
Let $\xi$ be a real vector bundle over $M$. 
Suppose that there exists $d\in H^{2}(M;\Z)$ such that $\rho_{2}(d)=w_{2}(\xi)$. 
Set $\xi^{\prime} = \xi - (l_{d})_{\R}$.
Since $c_{2i+1}(\xi^{\prime}\otimes\C)$ is torsion, it follows from the equation \eqref{eq:cheta} that
\begin{align*}
ch(\xi^{\prime}\otimes \C) - \dim_{\R} \xi^{\prime} = - c_{2}(\xi^{\prime} \otimes \C) + \frac{c_{2}^{2}(\xi^{\prime} \otimes \C)}{12} - \frac{c_{4}(\xi^{\prime} \otimes \C)}{6}
 = p_{1}(\xi^{\prime}) + \frac{ p_{1}^{2}(\xi^{\prime})} {12} - \frac{p_{2}(\xi^{\prime})}{6}.
\end{align*}
By construction, $\xi^{\prime}$ is spin. Thus
$p_{1}(\xi^{\prime}) = 2 q_{1}(\xi^{\prime})$ and $p_{2}(\xi^{\prime}) = 2 q_{2}(\xi^{\prime}) + q_{1}^{2}(\xi^{\prime})$
by the identities \eqref{eq:q1} and \eqref{eq:q2}. Therefore, 
$$ch(\xi^{\prime}\otimes \C) - \dim_{\R} \xi^{\prime} = p_{1}(\xi^{\prime}) + \frac{ p_{1}^{2}(\xi^{\prime})} {12} - \frac{p_{2}(\xi^{\prime})}{6} = 2 q_{1}(\xi^{\prime}) + \frac{q_{1}^{2}(\xi^{\prime})}{6} - \frac{q_{2}(\xi^{\prime})}{3}.$$
Consequently, 
\begin{align*}
3 \hat{A}(c, x, \xi^{\prime} \otimes \C )  =  &~ 3 \langle \hat{\mathfrak{A}}(M)\cdot e^{c/2}\cdot [ch( l_{x}) - 1] \cdot [ch(\xi^{\prime}\otimes\C) - \dim_{\R} \xi^{\prime}], ~[M]\rangle   \\
 = &~ \langle (x^{3} + \frac{3cx^{2}}{2} + \frac{3c^{2}x}{4} - \frac{p_{1}(M)x}{4})q_{1}(\xi^{\prime}) + x(\frac{q_{1}^{2}(\xi^{\prime})}{2} - q_{2}(\xi^{\prime})), ~[M] \rangle  \\
 = &~ \langle\frac{x}{2}\cdot q_{1}(\xi^{\prime}) \cdot(q_{1}(\xi^{\prime})-\frac{p_{1}(M)-c^{2}}{2}), ~[M]\rangle   \\
& +\langle q_{1}(\xi^{\prime}) (x^{2}+cx) (2x+c)/2, ~ M]\rangle 
  -\langle q_{2}(\xi^{\prime})\cdot x, ~ [M]\rangle, 
\end{align*}
which is an integer by Lemma \ref{lem:rr}. \qed
\end{example}

\begin{remark}
In Example \ref{exam:xi}, by construction, we have 
\begin{equation} \label{eq:q1p}
q_{1}(\xi^{\prime}) = \frac{p_{1}(M) - d^{2}}{2} = \frac{ p_{1}(\xi) - d^{2}}{2}.
\end{equation}
\end{remark}

\begin{example}\label{exam:7}
Let $\eta$ (resp. $\xi$) be a complex (resp. real) vector bundle over $M$. Suppose that $\zeta = \eta_{\R} - \xi$ is stably trivial over $M^{(7)}$. 
Then $\zeta \otimes \C$ is stably trivial over $M^{(7)}$. Moreover, since $c_{5}(\zeta \otimes \C)$ is torsion, it follows that 
$$ch(\zeta \otimes \C) = - \frac{c_{4}(\zeta \otimes \C)}{6} = - \frac{p_{2}(\zeta)}{6}.$$
Obviously, $\zeta$ is spin with $q_{1}(\zeta) = 0$.
Therefore, 
$$ch(\zeta \otimes \C) =  - \frac{p_{2}(\zeta)}{6} = - \frac{q_{2}(\zeta)}{3}$$
by the identities \eqref{eq:q1} and \eqref{eq:q2}. 
Thus, 
\begin{align*}
3 \hat{A}(c, x, \zeta \otimes \C) = &~3 \langle \hat{\mathfrak{A}}(M)\cdot e^{c/2}\cdot [ch(l_{x}) - 1] \cdot [ch(\zeta \otimes \C) - \dim_{\R} \zeta],~[M]\rangle \\
= &~ - \langle q_{2}(\zeta)\cdot x,~[M] \rangle,
\end{align*}
which is obviously an integer. \qed
\end{example}

\begin{lemma}\label{lem:dm}
$\mathcal{D}_{M}$ is the annihilator of $\mathrm{Sq}^{2}\rho_{2} ( H^{6}(M;\Z) )$ with respect to the cup product.
\end{lemma}

\begin{proof}
Recall that 
\[ \mathcal{D}_{M} = \{ x \in H^{2}(M; \Z/2) \mid x^{2} + w_{2}(M) x \in \rho_{2} ( TH^{4}(M;\Z) ) \}.\]
For any $x\in H^{2}(M;\Z/2)$ and any $z\in H^{6}(M;\Z)$, it follows from the Cartan formula and $V_{2}(M) = w_{2}(M)$ that
\begin{equation*}
x  \mathrm{Sq}^{2}\rho_{2}(z) = \mathrm{Sq}^{2}(x  \rho_{2}(z)) + \mathrm{Sq}^{2}(x)  \rho_{2}(z) = (w_{2}(M) x + x^{2})\cdot\rho_{2}(z).
\end{equation*}
Here $V_{2}(M)$ is the second Wu class of $M$. 
Since the annihilator of $H^{6}(M;\Z)$ is $TH^{4}(M;\Z)$ (cf. \cite[Lemma 1]{masseysw}), the fact of this lemma can be deduced easily from the definition of $\mathfrak{D}_{M}$.
\end{proof}

For any $x \in \mathcal{D}_{M}$, since $\mathcal{D}_{M} \subset \rho_{2}(H^{2}(M;\Z))$, we will take and fix one element $\tilde{x} \in H^{2}(M; \Z)$ such that $\rho_{2}(\tilde{x}) = x$.
Then it follows from the definition of $\mathcal{D}_{M}$ that there are $z_{c, \tilde{x}} \in H^{4}(M; \Z)$ and $t_{c,\tilde{x}} \in TH^{4}(M; \Z)$ such that
\begin{equation}\label{eq:zcx}
\tilde{x}^{2} + c\tilde{x} = 2 z_{c, \tilde{x}} + t_{c, \tilde{x}}.
\end{equation}
For a real vector bundle $\xi$ over $M$, and $d \in H^{2}(M;\Z)$, we will set
\begin{equation*}
A_{c,\xi}(d, \tilde{x}) : =\langle\frac{\tilde{x}}{2}\cdot\frac{p_{1}(\xi)-d^{2}}{2}\cdot (\frac{p_{1}(\xi)-d^{2}}{2}-\frac{p_{1}(M)-c^{2}}{2} ), ~[M]\rangle.
\end{equation*}

\begin{lemma}\label{lem:w8A}
Let $\eta$ be a complex vector bundle over $M$ such that $\eta_{\R} |_{M^{(7)}} $ is a stable complex structure of $\xi$ over $M^{(7)}$.
Then $w_{8}(\eta)-w_{8}(\xi) \in \mathrm{Sq}^{2}\rho_{2}H^{6}(M;\Z)$ if and only if 
\begin{equation*}
(w_{8}(\xi)+w_{2}(\xi)\mathrm{Sq}^{2}w_{4}(\xi)) \cdot  x \equiv A_{c,\xi}(c_{1}(\eta), \tilde{x}) + w_{4}(\xi) \mathrm{Sq}^{2} \rho_{2} (z_{c, \tilde{x}})  \mod 2
\end{equation*}
holds for every $x\in\mathcal{D}_{M}$. 
\end{lemma}

\begin{remark}\label{rem:acxi}
It follows from the proof of this lemma that the rational number $A_{c,\xi}(c_{1}(\eta), \tilde{x})$ is an integer, so it make sense to take congruent classes modulo $2$.
\end{remark}

\begin{remark}
It follows from the proof of this lemma that the $\bmod ~ 2$ value of 
$$A_{c,\xi}(c_{1}(\eta), \tilde{x}) + w_{4}(\xi) \mathrm{Sq}^{2} \rho_{2} (z_{c, \tilde{x}})$$
 does not depend on the choice of $c, \tilde{x}$ and the selection of $z_{c, \tilde{x}}$. 
\end{remark}

\begin{proof}
Set $\zeta = \eta_{\R} - \xi$. Then $w_{8}(\eta) - w_{8}(\xi) = \rho_{2}(q_{2}(\zeta))$ by the equation \eqref{eq:po8q}.
Therefore,
$$w_{8}(\eta)-w_{8}(\xi) \in \mathrm{Sq}^{2}\rho_{2}H^{6}(M;\Z)$$
if and only if 
$x \cdot \rho_{2}(q_{2}(\zeta)) = 0$,
i.e.,
$$\langle \tilde{x} \cdot q_{2}(\zeta), ~ [M] \rangle \equiv0~ \bmod 2$$
holds for every  $x \in \mathcal{D}_{M}$ by Lemma \ref{lem:dm}. 

Moreover, set $\eta^{\prime} = \eta -  l_{\eta}$ and $\xi^{\prime} = \xi - ( l_{\eta})_{\R}$.
Since
$$\langle \tilde{x} \cdot q_{2}(\zeta),~[M] \rangle = - 3 \hat{A}(c, \tilde{x}, \zeta \otimes \C)$$ by Example \ref{exam:7}, and
\begin{align*}
\hat{A}(c, \tilde{x}, \zeta \otimes \C) & = \hat{A}(c, \tilde{x}, \eta_{\R} \otimes \C) - \hat{A}(c, \tilde{x}, \xi \otimes \C) \\
& = \hat{A}(c, \tilde{x}, \eta_{\R}^{\prime} \otimes \C) - \hat{A}(c, \tilde{x}, \xi^{\prime} \otimes \C) \\
& = \hat{A}(c, \tilde{x}, \eta^{\prime}) +\hat{A}(c, \tilde{x}, \bar{\eta^{\prime}}) - \hat{A}(c, \tilde{x}, \xi^{\prime} \otimes \C) \\
& = \hat{A}(c, \tilde{x}, \eta^{\prime})  - \hat{A}(c, \tilde{x}, \bar{\eta^{\prime}}) - \hat{A}(c, \tilde{x}, \xi^{\prime} \otimes \C) + 2 \hat{A}(c, \tilde{x}, \bar{\eta^{\prime}}),
\end{align*}
It follows that
\begin{equation} \label{eq:A}
\langle \tilde{x} \cdot q_{2}(\zeta), ~ [M] \rangle \equiv \hat{A}(c, \tilde{x}, \eta^{\prime})  - \hat{A}(c, \tilde{x}, \bar{\eta^{\prime}}) - 3\hat{A}(c, \tilde{x}, \xi^{\prime} \otimes \C)  \bmod 2.
\end{equation}
Now, note that we have the following facts:
\begin{enumerate}
\item[(1)] it follows from Examples \ref{exam:eta} and \ref{exam:xi}, the identities \eqref{eq:zcx} and \eqref{eq:q1p} that
\begin{align*}
\hat{A}(c, \tilde{x}, \eta^{\prime})  - \hat{A}(c, \tilde{x}, \bar{\eta^{\prime}}) & =  \langle c_{3}(\eta^{\prime}) \cdot ( \tilde{x}^{2} + c \tilde{x})/2,~[M] \rangle \\
& = \langle c_{3}(\eta^{\prime}) \cdot z_{c,\tilde{x}}, ~[M] \rangle,\\
3 \hat{A}(c, \tilde{x}, \xi^{\prime} \otimes \C ) &  =  A_{c,\xi}(c_{1}(\eta), \tilde{x}) + \langle q_{1}(\xi^{\prime}) ( \tilde{x}^{2}+c \tilde{x}) ( 2 \tilde{x} + c )/2, ~ M]\rangle 
  - \langle q_{2}(\xi^{\prime})\cdot \tilde{x}, ~ [M]\rangle \\
  & = A_{c,\xi}(c_{1}(\eta), \tilde{x}) + \langle q_{1}(\xi^{\prime}) \cdot z_{c, \tilde{x}} \cdot c, ~[M] \rangle - \langle q_{2}(\xi^{\prime})\cdot \tilde{x}, ~ [M]\rangle \\ 
  & \quad + 2 \langle q_{1}(\xi^{\prime}) \cdot z_{c, \tilde{x}} \cdot x, ~[M] \rangle,
\end{align*}

\item[(2)] $A_{c,\xi}(c_{1}(\eta), \tilde{x})$ is an integer by  fact (1),

\item[(3)] by construction, $\xi^{\prime}$ is spin and
\begin{align*}
\rho_{2}(q_{1}(\xi^{\prime})) & = w_{4}(\xi^{\prime}) = w_{4}(\xi), \\
\rho_{2}(q_{2}(\xi^{\prime}))  & = w_{8}(\xi^{\prime}) =w_{8}(\xi) + w_{2}(\xi)w_{6}(\xi) + w_{2}^{2}(\xi)w_{4}(\xi) , \\
c_{3}(\eta^{\prime}) & = c_{3}(\eta) - c_{1}(\eta) c_{2}(\eta),
\end{align*}

\item[(4)] Since $\eta_{\R} |_{M^{(7)}} $ is a stable complex structure of $\xi$ over $M^{(7)}$, 
$\rho_{2}(c_{i}(\eta)) = w_{2i}(\xi)$ for $i \le 3$,

\item[(5)] It follows from the Wu's explicit formula (cf. \cite[Problem 8-A]{ms74}) that 
$$\mathrm{Sq}^{2}w_{4}(\xi)  = w_{2}(\xi) w_{4}(\xi) + w_{6}(\xi),$$

\item[(6)] $w_{2}(M)w_{4}(\xi) \rho_{2} (z_{c, \tilde{x}}) = \mathrm{Sq}^{2}(w_{4}(\xi) \rho_{2} (z_{c, \tilde{x}})) = \mathrm{Sq}^{2}w_{4}(\xi) \cdot \rho_{2}(z_{c, \tilde{x}}) + w_{4}(\xi)\mathrm{Sq}^{2}\rho_{2}(z_{c, \tilde{x}}).$
\end{enumerate}
Then substituting these facts into the congruence \eqref{eq:A}, we can get that
\begin{align*}
& ~\langle \tilde{x} \cdot q_{2}(\zeta), ~ [M] \rangle \\
\equiv & ~ \hat{A}(c, \tilde{x}, \eta^{\prime})  - \hat{A}(c, \tilde{x}, \bar{\eta^{\prime}}) - 3\hat{A}(c, \tilde{x}, \xi^{\prime} \otimes \C) \bmod  2\\
 \equiv &~(w_{8}(\xi)+w_{2}(\xi)\mathrm{Sq}^{2}w_{4}(\xi)) \cdot x + A_{c,\xi}(c_{1}(\eta), \tilde{x}) + w_{4}(\xi) \mathrm{Sq}^{2} \rho_{2} (z_{c, \tilde{x}})  \bmod  2,
\end{align*}
which completes the proof.
\end{proof}

\begin{theorem}\label{thm:csxi}
Let $\xi$ be a real vector bundle over $M$. 
Then $\xi$ admits a stable complex structure if and only if $\beta(w_{2}(\xi)) = 0$, $\beta (w_{6}(\xi)) = 0$ and
\begin{equation}\label{eq:xid}
(w_{8}(\xi)+w_{2}(\xi)\mathrm{Sq}^{2}w_{4}(\xi)) \cdot x \equiv A_{c,\xi}(d, \tilde{x}) + w_{4}(\xi) \mathrm{Sq}^{2} \rho_{2} (z_{c, \tilde{x}})  \mod 2
\end{equation}
holds for every $x\in\mathcal{D}_{M}$, where $d\in H^{2}(M;\Z)$ is an element satisfies $\rho_{2}(d)=w_{2}(\xi)$.
\end{theorem}

\begin{remark}
Note that $\beta(w_{2}(\xi)) = 0$ implies that there exists $d \in H^{2}(M; \Z)$ such that $\rho_{2}(d) = w_{2}(\xi)$ by the Bockstein sequence \eqref{eq:beta}.
\end{remark}

\begin{remark}
Suppose that $\beta(w_{2}(\xi)) = 0$ and $\beta (w_{6}(\xi)) = 0$. Then it follows from the proof of Theorem \ref{thm:csxi} that there is a complex vector bundle $\gamma$ over $M$ with $c_{1}(\gamma) = d$ such that $\gamma_{\R} |_{M^{(7)}} $ is a stable complex structure of $\xi$ over $M^{(7)}$. Therefore, 
the rational number $A_{c,\xi,}(d, \tilde{x})$ is an integer by Remark \ref{rem:acxi}, so it make sense to take congruent classes modulo $2$.
\end{remark}

\begin{remark}
Theorem \ref{thm:csxi} is a generalization of Dessai \cite[Theorem 1.9]{dessai}.
The idea of the proof of Lemma \ref{lem:w8A} and Theorem \ref{thm:csxi} is due to Dessai \cite{dessai}. 
One may find that the congruence \eqref{eq:xid} is a simplification of the congruence in \cite[Theorem 1.9]{dessai}. 
In the proof of \cite[Theorem 1.9]{dessai}, under the assumption that $H_{3}(M ; \Z)$ contains no $2$-torsion, 
a complex vector bundle $G_{x}$ was constructed to eliminate $ch(F) + ch(\bar{F})$ such that the congruence in \cite[Theorem 1.9]{dessai} was obtained (cf. \cite[p. 165, line 7]{dessai}). 
Our trick is different. 
With no assumptions on $H_{3}(M ; \Z)$, we use the identity $ch(\eta) + ch(\bar{\eta}) = ch(\eta) - ch(\bar{\eta}) + 2ch(\bar{\eta})$ to get the congruence \eqref{eq:xid} (cf. the equation \eqref{eq:A}).

\end{remark}


\begin{proof}
One direction follows immediately from Theorem \ref{thm:po8} and Lemmas \ref{lem:o90} and \ref{lem:w8A}.

Now we assume that $\beta(w_{2}(\xi)) = 0$, $\beta(w_{6}(\xi)) = 0$ and the congruence \eqref{eq:xid} holds for every $x \in \mathcal{D}_{M}$.
Then it follows from $\beta(w_{2}(\xi)) = 0$, $\beta(w_{6}(\xi)) = 0$ and Lemmas \ref{lem:o90} and \ref{lem:oc} that there exists a complex vector bundle $\eta$ over $M$ such that $\eta_{\R} |_{M^{(7)}} $ is a stable complex structure of $\xi$ over $M^{(7)}$.
Since $\rho_{2}(d) = \rho_{2}(c_{1}(\eta)) = w_{2}(\xi)$, there must exists $y \in H^{2}(M; \Z)$ such that $d = c_{1}(y) + 2y$.
Set $$\gamma = \eta + l_{y} - \bar{l}_{y}.$$ 
Then $\gamma$ is a complex vector bundle over $M$ with $c_{1}(\gamma) = d$ (cf. Heaps \cite[Lemma 2.1]{heaps}) and $\gamma_{\R} |_{M^{(7)}} $ is also a stable complex structure of $\xi$ over $M^{(7)}$.  
Therefore, note that the congruence \eqref{eq:xid} holds for every $x \in \mathcal{D}_{M}$, Lemma \ref{lem:w8A} tells us that
$$w_{8}(\gamma) - w_{8}(\xi) \in \mathrm{Sq}^{2}H^{6}(M: \Z).$$
Thus $\xi$ admits a stable complex structure by Theorem \ref{thm:po8}.
The proof is completed.
\end{proof}

\begin{remark}\label{rem:acxiw}
In the proof of Theorem \ref{thm:csxi}, we have $w_{8}(\gamma) = w_{8}(\eta)$. Then it follows from the proof of Lemma \ref{lem:w8A} that the $\bmod  ~2$ value of $A_{c,\xi}(d, \tilde{x}) + w_{4}(\xi) \mathrm{Sq}^{2} \rho_{2} (z_{c, \tilde{x}})$, does not depend on the choice of $c, d, \tilde{x}$ and the selection of $z_{c, \tilde{x}}$.
\end{remark}

The rest of this section is devoted to prove Theorem \ref{thm:main}.

Denote by $V_{i}(M) \in H^{i}(M;\Z/2)$ the Wu-class of $M$. It is known that they satisfy the Wu formula (cf. \cite[p. 132]{ms74})
\begin{equation}\label{eq:wu}
w_{k}(M)=\Sigma_{i=0}^{k}\mathrm{Sq}^{i}V_{k-i}(M).
\end{equation}
Since $M$ is orientable, i.e., $w_{1}(M) = 0$, it follows from the definition of $V_{i}(M)$ and $\mathrm{Sq}^{2i+1} = \mathrm{Sq}^{1}\mathrm{Sq}^{2i}$ that $V_{i}(M)=0$ for $i\neq2,~4$. Moreover, note that $M$ is spin$^{c}$, it can be deduced easily from the Wu formula \eqref{eq:wu} that
\begin{align}
V_{2}(M)&=w_{2}(M), \notag \\
V_{4}(M) & =w_{4}(M)+w_{2}^{2}(M),\label{eq:v4} \\
w_{6}(M)&=\mathrm{Sq}^{2}V_{4}(M)= \mathrm{Sq}^{2}w_{4}(M),\label{eq:w6} \\
w_{8}(M)&=\mathrm{Sq}^{4}V_{4}(M) = V_{4}^{2}(M)=w_{4}^{2}(M)+w_{2}^{4}(M).\label{eq:w8}
\end{align}
Furthermore, from Wu's explicit formula (cf. \cite[Problem 8-A]{ms74}), we have
$$\mathrm{Sq}^{2}w_{4}(M)=w_{2}(M)w_{4}(M)+w_{6}(M).$$
Therefore, combining this identity with the identity \eqref{eq:w6}, yields
\begin{equation} \label{eq:w24}
w_{2}(M)w_{4}(M)=0.
\end{equation}
Thus,
\begin{equation}\label{eq:w26}
w_{2}(M)w_{6}(M) = w_{2}(M) \mathrm{Sq}^{2}w_{4}(M) = \mathrm{Sq}^{2}(w_{2}(M)w_{4}(M)) + w_{2}^{2}(M) w_{4}(M) = 0.
\end{equation}

\begin{lemma}\label{lem:a}
For any $x \in \mathcal{D}_{M}$ and $a,~b\in H^{2}(M;\Z)$, we have $x \cdot \rho_{2}(a^{2}b^{2})=0$. In particular, $x \cdot \rho_{2}(a^{4})=0$.
\end{lemma}
\begin{proof} 
Since $\rho_{2}(a^{2}b^{2}) = \mathrm{Sq}^{2}\rho_{2}(ab^{2}) \in \mathrm{Sq}^{2}\rho_{2} (H^{6}(M;\Z))$, the facts of this lemma follow directly from Lemma \ref{lem:dm}.
\end{proof}

\begin{lemma}\label{lem:ac}
For any $x \in \mathcal{D}_{M}$ and  $d\in H^{2}(M;\Z)$ which satisfies $\rho_{2}(d) = w_{2}(M)$, 
$$A_{c,TM}(d, \tilde{x})\equiv0 \bmod 2.$$
\end{lemma}
\begin{proof}
Since $\rho_{2}(c) = \rho_{2}(d) = w_{2}(M)$, there must exists $a \in H^{2}(M; \Z)$ such that $d = c- 2a$. Then by definition, 
\begin{align*}
 A_{c,TM}(d, \tilde{x}) & = \langle\frac{\tilde{x}}{2}\cdot\frac{p_{1}(M)-d^{2}}{2}\cdot(\frac{p_{1}(M)-d^{2}}{2}-\frac{p_{1}(M)-c^{2}}{2}), ~[M]\rangle \\
 & = \langle\frac{\tilde{x}}{2}\cdot \frac{p_{1}(M) - d^{2}}{2}\cdot(\frac{p_{1}(M)-(c-2a)^{2}}{2}-\frac{p_{1}(M)-c^{2}}{2}), ~[M]\rangle \\
 & = \langle \tilde{x} \cdot (ac - a^{2}) \cdot \frac{p_{1}(M) - d^{2}}{2}, ~[M] \rangle.
 \end{align*}
 Now, since we have $\rho_{2}(\frac{p_{1}(M)-d^{2}}{2})=w_{4}(M)$, and
\begin{align*}
& ~\rho_{2} \left(\tilde{x} \cdot (ac - a^{2}) \cdot \frac{p_{1}(M) - d^{2}}{2} \right) \\
= & ~ x  w_{2}(M)w_{4}(M)  \rho_{2}(a) - x w_{4}(M) \rho_{2}(a^{2}) \\
 = &~ x w_{4}(M) \rho_{2}(a^{2}) \text{\quad by the identity \eqref{eq:w24}}\\
 = &~ \mathrm{Sq}^{4}(x \rho_{2}(a^{2})) + x w_{2}^{2}(M)\rho_{2}(a^{2}) \text{\quad by the identities \eqref{eq:wu} and \eqref{eq:v4}}\\
 =&~ x \rho_{2}(a^{4}) + x \rho_{2}(c^{2}a^{2}) \text{\quad by the Cartan formula}\\
 = & ~0\text{\quad by Lemma \ref{lem:a}},
\end{align*}
it follows that we always have
$A_{c,TM}(d, \tilde{x})\equiv0 \bmod 2$, which completes the proof.
\end{proof}

We are now in a position to prove Theorem \ref{thm:main}.

\begin{proof}[Proof of Theorem \ref{thm:main}]
Since $\mathcal{D}_{M} \subset \rho_{2} (H^{2}(M; \Z))$ implies that $M$ is spin$^{c}$, $\beta(w_{2}(M)) = 0$ is always satisfied.
Moreover, we have $x w_{2}^{4}(M) = 0$ by Lemma \ref{lem:a}, $w_{2}(M)\mathrm{Sq}^{2}w_{4}(M) = w_{2}(M) w_{6}(M) = 0$ by the identities \eqref{eq:w6} and \eqref{eq:w26}, and 
\begin{align*}
 w_{4}(M)\mathrm{Sq}^{2}\rho_{2}(z_{c,\tilde{x}}) & = \mathrm{Sq}^{2}(w_{4}(M) \rho_{2}(z_{c,\tilde{x}})) + \rho_{2}(z_{c,\tilde{x}})\mathrm{Sq}^{2}w_{4}(M) \\
& = w_{2}(M)w_{4}(M)\rho_{2}(z_{c,\tilde{x}}) + \rho_{2}(z_{c,\tilde{x}})\mathrm{Sq}^{2}w_{4}(M) \\
& = \rho_{2}(z_{c,\tilde{x}})\mathrm{Sq}^{2}w_{4}(M) \\
& = \rho_{2}(z_{c,\tilde{x}})w_{6}(M)
\end{align*}
by the identities \eqref{eq:w6} and \eqref{eq:w24}.
Then combing this facts with Lemma \ref{lem:ac} and the identity \eqref{eq:w8}, the statements of Theorem \ref{thm:main} follows from Theorem \ref{thm:csxi} by taking $\xi = TM$.
\end{proof}

%
%

\noindent
\thanks{\textbf{Acknowledgment.} The author would like to thank the University of Melbourne where parts of this work to be carried out and also Diarmuid Crowley for his hospitality. }














%


\end{document}